\documentclass[11pt,fleqn]{amsart} 

\usepackage[usenames,dvipsnames,condensed]{xcolor}

\usepackage{amsthm}
\usepackage{amsfonts}
\usepackage[english]{babel}
\usepackage[usenames]{xcolor}
\usepackage{graphicx}
\usepackage{soul}
\usepackage{stfloats}
\usepackage{morefloats}
\usepackage{cite}
\usepackage{lscape}
\usepackage{epstopdf}
\usepackage{bbm}
\usepackage[lite]{amsrefs}
\usepackage{tikz-cd}
\usepackage{tikz}
\usepackage{shuffle}
\usepackage{lscape}
\usepackage{rotating}
\usepackage{amsmath,amstext,amsopn,amsfonts,eucal,amssymb}
\usepackage{graphicx,wrapfig,url}

\newcommand\N{{\mathbb N}}

\newtheorem{theorem}{Theorem}[section]
\newtheorem{corollary}[theorem]{Corollary}
\newtheorem{lemma}[theorem]{Lemma}
\newtheorem{proposition}[theorem]{Proposition}
\newtheorem{definition}[theorem]{Definition}

\newtheorem{example}[theorem]{Example}

\newtheorem{remark}[theorem]{Remark}

\newcommand{\lt}{ \triangleleft }

\begin{document}
	\title[Lie TSD objects]{$3$-Lie Algebras, Ternary Nambu-Lie algebras and the Yang-Baxter equation}
	
	\author{Viktor Abramov} 
	\address{Institute of Mathematics and Statistics, University of Tartu\\
		Narva mnt 18, 51009 Tartu, Estonia} 
	\email{viktor.abramov@ut.ee}
	
	\author{Emanuele Zappala}
	\address{Institute of Mathematics and Statistics, University of Tartu\\
	Narva mnt 18, 51009 Tartu, Estonia.
	Current address: Yale University} 
	\email{zae@usf.edu\\ emanuele.zappala@yale.edu}
	
	\maketitle
	
	\begin{abstract}
		We construct ternary self-distributive (TSD) objects from compositions of binary Lie algebras, $3$-Lie algebras and, in particular, ternary Nambu-Lie algebras. We show that the structures obtained satisfy an invertibility property resembling that of racks. We prove that these structures give rise to Yang-Baxter operators in the tensor product of the base vector space and, upon defining suitable twisting isomorphisms, we obtain representations of the infinite (framed) braid group. We consider examples for low-dimensional Lie algebras, where the ternary bracket is defined by composition of the binary ones, along with simple $3$-Lie algebras. We show that the Yang-Baxter operators obtained are not gauge equivalent to the transposition operator, and we consider the problem of deforming the operators to obtain new solutions to the Yang-Baxter equation. We discuss the applications of this deformation procedure to the construction of (framed) link invariants. 
	\end{abstract}


\tableofcontents

	\section{Introduction}
	
	Yang-Baxter operators, i.e. solutions to the Yang-Baxter (YB), or braid, equation, have long been known to produce invariants of knots and links, and $3$-manifold invariants  \cite{Oht,Tur}. More recently, self-distributive structures such as racks and quandles have been shown to be suitable to construct link and knotted surface invariants \cite{CJKLS}, as well as manifold invariants \cite{Yetter}, via their cohomology theories. On the other hand, ternary and higher operations have been employed in physics to generalize Hamiltonian mechanics \cite{Nambu} and to derive related field theories in theoretical physics \cite{BL,BL2,DFMR}.
	
	 The purposes of this article are two-fold. First, solutions to the Yang-Baxter equation will be derived from ternary Nambu-Lie algebras. Second, link invariants are constructed from such solutions. The main step in the construction is that of obtaining ternary self-distributive (TSD) objects in the category of vector spaces, in the sense of \cite{ESZ} Section~8, satisfying extra conditions that are sufficient to establish the invertibility of the associated operators. These properties are a vector space analogue of the axioms for a ternary {\it rack}. In the binary case, it is well known that a rack, e.g. a {\it quandle}, produces solutions of the YBE by linearization, explicitly the operator is obtained from $x\otimes y \mapsto y\otimes x*y$. A doubled construction, from set-theoretical results found in \cite{ESZ}, allows to define a YB operator on $\mathbbm k\langle X\rangle \otimes \mathbbm k\langle X\rangle$, where $X$ is a ternary rack and $\mathbbm k\langle X\rangle$ denotes the linearization of the set $X$ over $\mathbbm k$. Repetitions of elements in this case are to be thought of as applying a comultiplication that linearizes the map $x\mapsto x\otimes x$. In general, there exist TSD objects in symmetric monoidal categories, such as the category of vector spaces, whose comultiplication is more general than the one given above. 
	For instance, let us consider the heap operation of a group $G$, defined as the map $(x,y,z)\mapsto xy^{-1}z$, where juxtaposition denotes multiplication in $G$. Then, this operation is known to be TSD, and its linearization (on the group algebra) defines a TSD object in the category of vector spaces. More generally, in a Hopf algebra, extending the map $x\otimes y\otimes z \mapsto xS(y)z$ by linearity, where $S$ denotes the antipode, defines a TSD object with respect to the comultiplication of the Hopf algebra. The heap map, in addition, is known to be indecomposable into binary SD operations \cite{ESZ}, and it is therefore an SD operation that is not obtained by composition of lower arity operations.
	
	Along with YB operators, we show that the TSD objects studied herein naturally give a notion of twisting morphism which is compatible with the YB operators in the sense that they commute. As a consequence we can derive a representation of the infinite framed braid group which allows us to define associated framed link invariants. 
	
	 The fundamental paradigm employed here is the following. We start with a coalgebra object $(X,\Delta)$ in a linear symmetric monoidal category, endowed with the extra structrure of a Lie algebra object. We construct a TSD operation using the Lie bracket and, making use of comultiplication and symmetry morphisms, we show that this TSD operation gives rise to a Yang-Baxter operator. Observe that the notion of TSD operation contains, by definition, a compatibility condition between comultiplication, symmetry morphisms and Lie bracket as well. Proper definitions will be provided in Section~\ref{sec:pre}. For the sake of simplicity, throughout this article, we restrict our attention to the case of the category of vector spaces. The results of this paper can be generalized to arbitrary linear symmetric monoidal categories in a fashion similar to \cite{EZ}, where the proofs have been carried out for heap objects in arbitrary symmetric monoidal categories. However, to have link invariants further axioms on the underlying category are required (e.g. existence of the trace), as in \cite{EZ}.
	
	This article, in fact, studies and widely expands a class of examples considered in \cite{EZ}, where Lie algebras with ternary operations defined by compositions of brackets are seen to provide TSD objects in the symmetric monoidal category of vector spaces. Here, we show that $3$-Lie algebras, and in particular ternary Nambu-Lie algebras, are also suitable for the purpose of constructing TSD objects. 
	
	We note that our algebraic results do not require further regularity assumptions, such as simplicity or semi-simplicity of the Lie algebras, but only depend on the defining axioms of (Nambu-)Lie algebras. In fact, one can see that the TSD objects of this article determine a representation of the framed braid group also when no assumption is made on the dimensionality of the underlying Lie algebra. In other words, the ternary (Nambu-)Lie algebra we start with, can be taken to be infinite.
%

	The article is organized as follows. In Section~\ref{sec:pre} we recall basic definitions that will be used throughout the paper. In Section~\ref{sec:LieTSD} we consider the case of TSD objects associated to compositions of binary Lie brackets, and construct the associated YB operators. Section~\ref{sec:ternaryLie} is dedicated to the case of ternary (Nambu-)Lie algebras and their corresponding YB operators. In Section~\ref{sec:invariants} we define twisting morphisms, for the TSD objects given in the previous sections, such that to obtain representations of the infinite framed braid group $\mathbb{FB}_\infty$. We consider the problem of deforming the YB operators and the effect of such deformations on the associated framed link invariants in some special cases for low-dimensional Lie algebras (with TSD operation induced by composition of binary operations and by ternary operations), and some links with small number of crossings.  

\bigskip 
{\noindent{\bf Acknowledgements.} EZ has been funded by the Estonian Research Council through the grant: MOBJD679. The authors would like to thank M. Saito and M. Elhamdadi for useful suggestions, and the referee for valuable comments that have helped to improve the article significantly.}
	
	\section{Preliminaries}\label{sec:pre}
	
	In this section we review some relevant material, and set some notation for the rest of the article. 
	
	\subsection{$3$-Lie algebras and ternary Nambu-Lie algebras}
	
	Let $X$ be a vector space over the field $\mathbbm k$ and let $[\bullet, \bullet, \bullet] : X^{\times 3} \longrightarrow X$ denote a trilinear map, where $\bullet$ is a placeholder. We say that $X$ is a $3$-Lie algebra, and that $[\bullet, \bullet, \bullet]$ is a ternary Lie bracket, if the equations
	$$
	[x_1, x_2, x_3] = (-1)^{|\sigma|} [x_{\sigma(1)}, x_{\sigma(2)}, x_{\sigma(3)}] 
	$$
	and 
	\begin{eqnarray*}
	\lefteqn{[[x_1,x_2,x_3],x_4,x_5]}\\
	&=& [[x_1,x_4,x_5],x_2,x_3] + [x_1,[x_2,x_4,x_5],x_3] + [x_1,x_2,[x_3,x_4,x_5]]
	\end{eqnarray*}
	hold for any permutation $\sigma\in \mathbb S_3$ and all $x_1, x_2, x_3, x_4, x_5\in X$, where $|\sigma|$ is the sign of $\sigma$. In other words, the bracket is skew-symmetric, and it satisfies the so called {\it Filippov identity}, which is a generalization of the ``usual'' Jacobi identity. It is possible to generalize this definition to the case of $n$-ary Lie algebras, i.e. $n$-Lie algebras. These structures, were introduced by Filippov in \cite{Filippov}. A special class of $n$-Lie algebras has been considered by Nambu in \cite{Nambu} to generalize the classical formulation of Hamiltonian Mechanics to $n$-ary operations. The bracket used by Nambu is defined as the Jacobian of $n$ functions, and these structures are now often referred to as {\it Nambu-Lie algebras}. Explicitly, let $X$ denote an associative algebra over a unital ring $R$. Let us fix derivations $D_1, \ldots , D_n$ of $X$, such that they commute with each other. Let $x_1, \ldots ,x_n$ be elements of $X$, then the $n$-bracket is defined as 
	$$
	[x_1, \ldots , x_n] = J(x_1, \ldots , x_n),
	$$ 
	where $J(x_1, \ldots, x_n)$ is the determinant of the matrix $C = (c_{ij})$ defined by $D_ix_j = c_{ij}$.  The fact that the $n$-bracket so defined satisfies the Filippov identity, and it therefore is an $n$-Lie algebra, is not straightforward. A proof can be found in \cite{Filippov}, Proposition~2. 
	
	
	While in the original article of Nambu it was not shown that the Filippov identity was satisfied, these structures were  the main class of examples of Filippov in \cite{Filippov}. We observe that other authors sometimes use different names to indicate $n$-Lie algebras and Nambu-Lie algebras. So, our nomenclature might not coincide with that of other references. 
	
	Over algebraically closed fields of characteristic zero, simple $n$-Lie algebras have been completely classified \cite{Ling}, and it is known that they are the Lie algebras $A_{n+1}$ introduced by Filippov in \cite{Filippov} (see Proposition~1 therein). We observe moreover, as pointed out in \cite{Poz}, that most of the known simple $n$-Lie algebras are all subalgebras of Nambu-Lie algebras, in the sense that they can be embedded as subalgebras in the algebra of functions with the bracket induced by the Jacobian, as described above. In characteristic $p$, there exist examples of finite dimensional simple $n$-Lie algebras that differ from the simple ones in characteristic  zero. Examples of these structures can be found for instance in \cite{Poz}.
	
	\begin{remark}
		{\rm 
		Observe that the original definition of $n$-Lie algebra was given for modules over a unitary ring, instead of vector spaces as given above. In fact, although we refer to Lie algebras over a field in the rest of the article, one could replace field by unitary ring without incurring in any problems.
	}
	\end{remark}
	
	\subsection{Knot-theoretic notions}
	A {\it rack}, is a non-empty set $X$ endowed with a binary operation $*: X\times X\longrightarrow X$ satisfying the following axioms
	\begin{itemize}
		\item[(i)] 
	  For each $x\in X$, the right multiplication map $\bullet*x: X\longrightarrow X$, where $\bullet$ indicates a placeholder, is a bijection;
	  \item[(ii)] 
	  Self-Distributivity: for all $x,y,z\in X$, we have $(x*y)*z = (x*z)*(y*z)$. 
	\end{itemize}
If, moreover, $*$ is idempotent, i.e. it satisfies $x*x = x$ for all $x\in X$, then $(X,*)$ is called {\it quandle}.

\begin{remark}
	{\rm 
It is well known that the axioms of racks are an algebraic version of Reidemeister moves II and III, respectively, while indempotence corresponds to Reidemeister move I. 	
}
\end{remark}

Quandles have been studied intensively from the 1980's, after Joyce \cite{Joy} and Matveev \cite{Mat} have independently shown that the {\it fundamental quandle} of a knot, defined via generators and relators using the Wirtinger presentation of the fundamental group, is a complete knot invariant (up to mirror image and orientation reversal). More recently, in \cite{CJKLS}, a cohomology theory of racks and quandles has been introduced, and a cohomological invariant of links and knotted surfaces has been constructed. 

The theory of racks/quandles and their cohomology theories have been generalized to operations of higher arities, see for instance \cite{Gre,ESZ}. Moreover, the notion of self-distributivity has been treated, in \cite{ESZ}, in arbitrary symmetric monoidal categories. Our main interest lies, in the present article, in self-distributive (SD) structures in the category of vector spaces. We will therefore focus on this class of SD operations. 

 Let $(X,\Delta, \epsilon)$ be a coalgebra in the category of vector spaces, and let $\Delta_3 = (\Delta\otimes \mathbbm 1)\circ \Delta$.
Recall that, a coalgebra $(X,\Delta, \epsilon)$ is a ternary self-distributive (TSD) object, if it is endowed with a coalgebra morphism $T: X\otimes X\otimes X\longrightarrow X$ such that the diagram 
\begin{center}
	\begin{tikzcd}
		&X^{\otimes 9}\arrow{dl}[swap]{\shuffle} & &X^{\otimes 5}\arrow{ll}[swap]{\mathbbm 1^{\otimes 3}\otimes \Delta_3^{\otimes 2}}\arrow {rd}{T\otimes \mathbbm{1}^{\otimes 2}} &  \\
		X^{\otimes 9}\arrow{dd}[swap]{T\otimes T \otimes T} & & & & X^{\otimes 3}\arrow{dd}{T}\\
		& & & &\\
		X^{\otimes 3}\arrow{rrrr}[swap]{T}& & & &X 
	\end{tikzcd} 
\end{center} 
commutes, where the morphism $\shuffle$ is defined as follows. On simple tensors of $X^{\otimes 9}$ it permutes the entries according to the permutation $\sigma =  (2,4)(3,7)(6,8)$, and it is therefore extended by linearity. 

In \cite{ESZ}, examples of TSD objects in vector spaces associated to (involutory) Hopf algebras via quantum heap or quantum conjugation operations, and to Lie algebras via composing certain binary SD operations originally found in \cite{CCES}, were introduced. The latter class of examples is recalled in detail in the next section, as it provides the starting point of the present article. 

\subsection{Framed Reidemeister moves and framed braid group}

A framing of a link $\mathcal L$ is a choice of a section of the normal bundle of $\mathcal L$. Framed links are represented diagrammatically by regular projections on the plane, where we thicken the arcs of the link to two parrallel arcs delimiting a ribbon. 
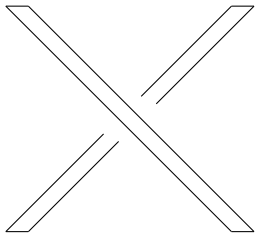
\begin{figure}
	\begin{center}
		\begin{tikzpicture}
		\draw (0,3) -- (3,0);
		\draw (0.3,3) -- (3.3,0);
		
		\draw (3,3) -- (1.8,1.8);
		\draw (1.3,1.3) -- (0,0);
		
		\draw (3.3,3) -- (2,1.7);
		\draw (1.5,1.2) -- (0.3,0);
		
		\draw (0,3) -- (0.3,3);
		\draw (3,3) -- (3.3,3);
		\draw (0,0) -- (0.3,0);
		\draw (3,0) -- (3.3,0);
		
		\end{tikzpicture}
	\end{center}
	\caption{Crossing between ribbon edges}
	\label{fig:crossing}
\end{figure}
 A crossing where the edges are thickened is shown in Figure~\ref{fig:crossing}. 
%
%
%
%
%
%
%
%
	 The isotopy class of a framed link is characterized by a set of combinatorial moves called {\it framed Reidemeister moves} which we hereby briefly mention. See Theorem~1.8 and Figure~1.8 in \cite{Oht} for further details. Reidemeister moves II and III are unchanged for framed links upon thickening, for example, arcs of a diagram. Reiedemister move I is replaced by the annihilation of kinks, where positive (resp. negative) self-crossings followed by negative (resp. positive) self-crossings can be replaced by a straight arc. 
%
%
%
%
%
%
%


The main objective is to obtain a representation of the framed braid group, whose definition we briefly recall. 
Let $\mathbb B_n$ denote the braid group on $n$-strings, presented by $n-1$ generators $\sigma_i$ with relations $\sigma_i\sigma_j = \sigma_j\sigma_i$ whenever $|i-j| \geq 2$, and $\sigma_i\sigma_{i+1}\sigma_i = \sigma_{i+1}\sigma_i\sigma_{i+1}$ for all $i$.  The framed braid group on $n$ strings (\cite{KS}), which we denote by the symbol $\mathbb {FB}_n$, is the semi-direct product $\mathbb Z^n\rtimes \mathbb B_n$. Explicitly, this group is presented by $2n-1$ generators $\sigma_1, \ldots , \sigma_{n-1}$, and $t_1, \ldots t_n$, where the generators $\sigma_i$ satisfy the same relations of the braid group, $t_it_j = t_jt_i$ for all $i,j$, and $\sigma_j t_i = t_{\tau_j(i)}\sigma_j $, where $\tau_j$ is the transposition $(j\ j+1)$. 

There is a framed version of Markov's theorem, asserting that every framed link can be presented as the closure of a framed braid. This is considered for instance in Lemma~1 of \cite{KS}. 

\subsection{Yang-Baxter operators}

Let $X$ be a vector space and let $R: X\otimes X\longrightarrow X\otimes X$ be a linear map on the tensor product. Then, we say that $R$ is a pre-Yang-Baxter operator if it satisfies the following equation
$$
(R\otimes \mathbbm 1)\circ (\mathbbm 1\otimes R)\circ (R\otimes \mathbbm 1) =(\mathbbm 1\otimes R)\circ (R\otimes \mathbbm 1)\circ (\mathbbm 1\otimes R),
$$
called Yang-Baxter equation (YBE), or braid equation. If, moreover, $R$ is invertible, then we say that $R$ is a Yang-Baxter operator. The braid equation is an algebraic formulation of Reidemeister move III, and it has also appeared in relation to particle scattering in physics \cite{Jimbo}. 

	\section{TSD objects associated to Lie algebras}\label{sec:LieTSD}
	
	Composing (binary) self-distributive objects in a symmetric monoidal category results into ternary self-distributive (TSD) objects as follows. Let $(X,\Delta,q)$ be a self-distributive object in the symmetric monoidal category $\mathcal C$, where $q: X\otimes X \longrightarrow X$ satisfies the categorical distributivity condition \cite{ESZ}, Section~$8$. Let us define $T: X\otimes X\otimes X \longrightarrow X$ as the composition $T := q\circ (q\otimes \mathbbm 1)$. Then (Theorem~$8.6$ in \cite{ESZ}) $T$ endows $X$ with the structure of a TSD object with respect to the same comultiplication of $(X,q)$. We describe a special case of this construction when the object $(X,\Delta,q)$ has map $q$ induced by the bracket of a Lie algebra, and the resulting TSD object is obtained, consequently, by means of binary brackets and nested binary brackets. The original binary construction was given in \cite{CCES}, while the ternary version appeared in \cite{ESZ}.
	
	Let $L$ be a Lie algebra over the ground field $\mathbbm k$, and define $X = \mathbbm k\oplus L$. 
	$X$ becomes a comonoid object in the category of vector spaces if we define a comultiplication $\Delta$ as
	$$
	(a,x) \mapsto (a,x)\otimes (1,0) + (1,0) \otimes (0,x),
	$$
	and counit $\epsilon$ as 
	$$
	(a,x)\mapsto a.
	$$
	
	Now we define $q: X\otimes X \longrightarrow X$ as 
	$$
	(a,x)\otimes (b,y) \mapsto (ab,bx + [x,y]).
	$$
	Direct computation (\cite{CCES}) using the Jacobi identity shows that $q$ satisfies the self-distributivity condition. The associated TSD object $(X,\Delta, T)$ has morphism $T$ defined by the assignment 
	$$
	(a,x)\otimes (b,y)\otimes (c,z) \mapsto (abc, bcx + c[x,y] + b[x,z] + [[x,y],z]),
	$$
	where $T$ satisfies the TSD property because it is obtained by composition of a self-distributive morphism, see Example~8.8 in \cite{ESZ}.

	\begin{remark}
		{\rm 
		By direct computation one obtains that the comultiplication $\Delta_3 : = (\Delta\otimes \mathbbm 1)\circ\Delta = (\mathbbm 1\otimes \Delta)\circ\Delta$, where the second equality is simply coassociativity, is given on simple tensors as
		$$
		(a,x) \mapsto (a,x)\otimes (1,0)\otimes (1,0) + (1,0)\otimes (0,x)\otimes (1,0) + (1,0)\otimes (1,0)\otimes (0,x).
		$$
	}
	\end{remark}


	\subsection{Solutions to the braid equation associated to Lie algebras}\label{subsec:braid}
	
	We use the TSD object induced by a (binary) Lie algebra structure to define a solution $R$ to the braid equation (i.e. a pre-Yang-Baxter operator) on $X\otimes X$ as follows. Let $L$ and $X$ be defined as above, and denote by $T$ the resulting TSD morphism. Let us use Sweedler's notation for comultiplication in the form $\Delta(a,x) = (a,x)^{(1)}\otimes (a,x)^{(2)}$. On simple tensors we set $R((a,x)\otimes (b,y)\otimes (c,z)\otimes (d,w))$ to be 
	\begin{eqnarray*}
	\lefteqn{R((a,x)\otimes (b,y)\otimes (c,z)\otimes (d,w))} \\
	&=& (c,z)^{(1)}\otimes (d,w)^{(1)} \otimes T((a,x)\otimes (c,z)^{(2)}\otimes (d,w)^{(2)})\otimes\\
	&&\hspace{0.5cm} \otimes T((b,y)\otimes (c,z)^{(3)}\otimes (d,w)^{(3)}) \in X^{\otimes 2}\otimes X^{\otimes 2}.
	\end{eqnarray*}
    
    Let us set some notation. Let $(X,\Delta)$ be a comonoid in a symmetric monoidal category $\mathcal C$. Then, define $\Delta_n: X\longrightarrow X^{\otimes n}$, for all $n\in \N_{\geq 2}$, to be $\Delta_n := (\mathbbm 1^{\otimes n-1}\otimes \Delta)\circ \cdots \circ(\mathbbm 1\otimes \Delta)\circ\Delta$. Clearly, due to coassociativity of $\Delta$ one could also equivalently define different compositions where $\Delta$ does not appear on the rightmost entry. For simplicity we will always consider the form of $\Delta_n$ given above. Coassociativity, in Sweedler notation, takes the following form. For $x\in X$, one has $\Delta_3(x) = (\mathbbm 1\otimes \Delta)\Delta(x) = (\mathbbm 1\otimes \Delta)(x^{(1)}\otimes x^{(2)}) = x^{(1)}\otimes x^{(21)}\otimes x^{(22)} = x^{(1)}\otimes x^{(2)} \otimes x^{(3)}$, where the last equality simply states that the total number of copies of $x$ is what really counts in the iterated application of comultiplication, rather than the way comultiplication is applied. For higher $n$, similar considerations hold.

	\begin{theorem}\label{thm:YBTSD}
		The map $R: X^{\otimes 2}\otimes X^{\otimes 2} \longrightarrow X^{\otimes 2}\otimes X^{\otimes 2}$ satisfies the braid equation. In other words, $R$ is a pre-Yang-Baxter operator on $X^{\otimes 2}\otimes X^{\otimes 2}$.
	\end{theorem}

\begin{proof}
	It is enough to show that the braid equation
	$$
	(R\otimes \mathbbm 1)\circ (\mathbbm 1\otimes R)\circ (R\otimes \mathbbm 1) = (\mathbbm 1\otimes R)\circ (R\otimes \mathbbm 1)\circ (\mathbbm 1\otimes R)
	$$
	holds on simple tensors. To simplify notation, we denote elements of $X$ by the letters $x,y,z$ etc, although elements of $X$ are pairs where one element is in the ground field $\mathbbm k$, and one element is in the Lie algebra $L$. Let us consider the LHS of the braid equation on a simple tensor $x\otimes y \otimes z\otimes w\otimes u\otimes v$. We have
	\begin{eqnarray*}
		\lefteqn{(R\otimes \mathbbm 1)\circ (\mathbbm 1\otimes R)\circ (R\otimes \mathbbm 1)(x\otimes y \otimes z\otimes w\otimes u\otimes v)}\\
		&=& u^{(11)}\otimes v^{(11)}\otimes T(z^{(1)}\otimes u^{(12)}\otimes v^{(12)})\otimes T(w^{(1)}\otimes u^{(13)}\otimes v^{(13)})\otimes\\
		&&\hspace{0.5cm} \otimes T(T(x\otimes z^{(2)}\otimes w^{(2)})\otimes u^{(2)}\otimes v^{(2)})\otimes \\
		&&\hspace{0.5cm} \otimes T(T(y\otimes z^{(3)}\otimes w^{(3)})\otimes u^{(3)}\otimes v^{(3)})\\
		&=& u^{(1)}\otimes v^{(1)}\otimes T(z^{(1)}\otimes u^{(2)}\otimes v^{(2)})\otimes T(w^{(1)}\otimes u^{(3)}\otimes v^{(3)})\otimes\\
		&&\hspace{0.5cm} \otimes T(T(x\otimes z^{(2)}\otimes w^{(2)})\otimes u^{(4)}\otimes v^{(4)})\otimes \\
		&&\hspace{0.5cm} \otimes T(T(y\otimes z^{(3)}\otimes w^{(3)})\otimes u^{(5)}\otimes v^{(5)})
		\end{eqnarray*}
	where the first equality uses the definition of $R$ and the fact that $T$ is a coalgebra morphism, while the second equality is a consequence of the  coassociativity of $\Delta$. Similarly, we obtain for the RHS 
	
	\begin{eqnarray*}
		\lefteqn{(\mathbbm 1\otimes R)\circ (R\otimes \mathbbm 1)\circ (\mathbbm 1\otimes R)(x\otimes y \otimes z\otimes w\otimes u\otimes v)}\\
		&=&(\mathbbm 1\otimes R)\circ (R\otimes \mathbbm 1)[x\otimes y\otimes u^{(1)}\otimes v^{(1)}\otimes T(z\otimes u^{(2)}\otimes v^{(2)})\otimes\\
		&& \hspace{1.5cm} T(w\otimes u^{(3)}\otimes v^{(3)})]\\
		&=& (\mathbbm 1\otimes R)[u^{(11)}\otimes v^{(11)}\otimes T(x\otimes u^{(12)}\otimes v^{(12)})\otimes T(y\otimes u^{(13)}\otimes v^{(13)})\\
		&& \otimes T(z\otimes u^{(2)}\otimes v^{(2)})\otimes T(w\otimes u^{(3)}\otimes v^{(3)})]\\
		&=& u^{(11)}\otimes v^{(11)}\otimes T(z^{(1)}\otimes u^{(21)}\otimes v^{(21)})\otimes T(w^{(1)}\otimes u^{(31)}\otimes v^{(31)})\otimes \\
		&& \otimes T(T(x\otimes u^{(12)}\otimes v^{(12)})\otimes T(z^{(2)}\otimes u^{(22)}\otimes v^{(22)})\otimes\\
		&&\hspace{1.5cm} T(w^{2}\otimes u^{(32)}\otimes v^{(32)}))\\
		&&\otimes T(T(y\otimes u^{(13)}\otimes v^{(13)})\otimes T(z^{(3)}\otimes u^{(23)}\otimes v^{(23)})\otimes\\
		&&\hspace{1.5cm} T(w^{3}\otimes u^{(33)}\otimes v^{(33)}))\\
		&=& u^{(1)}\otimes v^{(1)}\otimes T(z^{(1)}\otimes u^{(4)}\otimes v^{(4)})\otimes T(w^{(1)}\otimes u^{(7)}\otimes v^{(7)})\otimes \\
		&& \otimes T(T(x\otimes u^{(2)}\otimes v^{(2)})\otimes T(z^{(2)}\otimes u^{(5)}\otimes v^{(5)})\otimes\\
		&&\hspace{1.5cm} T(w^{2}\otimes u^{(8)}\otimes v^{(8)}))\\
		&&\otimes T(T(y\otimes u^{(3)}\otimes v^{(3)})\otimes T(z^{(3)}\otimes u^{(6)}\otimes v^{(6)})\otimes\\
		&&\hspace{1.5cm} T(w^{3}\otimes u^{(9)}\otimes v^{(9)})),
		\end{eqnarray*}
	where in the last equality we have applied the coassociativity of $\Delta$ to suppress double indices in Sweedler notation.
	Applying the TSD property to the RHS we see that the two terms differ by a reshuffling of the terms corresponding to the permutation
	$$
	\shuffle = \bigl(\begin{smallmatrix}
	1 & 2 & 3 & 4 & 5 &6 & 7 & 8 & 9 \\
	1 & 4 & 7 & 2 & 5 & 8 & 3 & 6 & 9
	\end{smallmatrix}\bigr).
	$$
	By direct inspection we see that $\Delta_3(a,x) = (\mathbbm 1\otimes \tau)\circ\Delta_3$, so that applying  the fact that $T$ is a morphism of comonoids, and therefore it commutes with the comultiplication, we see that the two terms given above coincide, which concludes the proof. 
\end{proof}

\subsection{Reversibility condition}

The TSD condition satisfied by $T$ allows to define a linear map $R$ on the doubled space $X^{\otimes 2}$ that satisfies the braid equation. We will discuss now an algebraic condition that is sufficient for $R$ to be invertible. 

Recall that if $(X,\mu, \eta,\Delta, \epsilon, S)$ is an involutory Hopf algebra, i.e. $S^2 = \mathbbm 1$, then there is a natural TSD structure on $X$ that generalizes the heap operation $(x,y,z) \mapsto xy^{-1}z$ in a group. Specifically, this is defined by the equation on simple tensors \cite{ESZ,ESZheap}
$$
x\otimes y\otimes z \mapsto xS(y)z.
$$

Moreover, the quantum heap structure assocaited to an involutory cocommutative Hopf algebra satisfies a condition, which we will call {\it reversibility condition}, that generalizes the set-theoretic equality $T(T(x,y,z),z,y) = x$ for the heap of a group. With the TSD structure defined above, reversibility takes the form

\begin{eqnarray*}
T(T(x\otimes y^{(2)}\otimes z^{(2)})\otimes z^{(1)}\otimes y^{(1)}) &=& T(xS(y^{(2)})z^{(2)}\otimes z^{(1)}\otimes y^{(1)})\\
&=&  xS(y^{(2)})z^{(2)}S(z^{(1)})y^{(1)}\\
&=& \epsilon(z) xS(y^{(2)})y^{(1)}\\
&=& \epsilon(y)\epsilon(z) x,
\end{eqnarray*}
where in the first and second equalities we have used the definition of $T$, while in the third and fourth equalities we have used the defining axiom relating comultiplication, multiplication and antipode in a (cocommutative) Hopf algebra. 

\begin{remark}
	{\rm 
	Observe that if we were to extend the group heap structure to the associated group algebra, then the two forms, set-theoretic and quantum heap, of reversibility would coincide, as a direct inspection shows.
}
\end{remark}

The reversibility condition plays a fundamental role, as it algebraically represents the Reidemeister move II. It is easy to see that, in general, the TSD structure associated to a Lie algebra as described at the beginning of Section~\ref{sec:LieTSD} does not satisfy the reversibility condition given above for quantum heaps. We pose the following definition.
\begin{definition}
	{\rm 
Let $(X,T)$ be a TSD object in the category of vector spaces. We say that $T$ is {\it reversible} if there exists a morphism $\tilde T: X\otimes X\otimes X\longrightarrow X$ such that $(X,\Delta,\tilde T)$ is a TSD object, and the pair $(T,\tilde T)$ satisfies the equation (reversibility condition)
$$
\tilde T(T(x\otimes y^{(2)}\otimes z^{(2)})\otimes z^{(1)}\otimes y^{(1)}) = \epsilon(y)\epsilon(z)\cdot x.
$$
A similar equality is also required when exchanging the roles of $\tilde T$ and $T$.  Moreover, we say that $T$ is reversible with respect to $\tilde T$. In this situation we say that $X$ is a {\it ternary rack} in the category of vector spaces. 
}
\end{definition}
We now show that, when $(X,T)$ arises from a Lie algebra $L$ as described above, $(X,T)$ is a ternary rack. Observe that the reversibility condition, with notation of pairs $(a,x)\in \mathbbm k\oplus L =: X$ is written 
$$
\tilde T(T((a,x)\otimes (b,y)^{(2)}\otimes (c,z)^{(2)})\otimes (c,z)^{(1)}\otimes (b,y)^{(1)}) = \epsilon(b,y)\epsilon(c,z)(a,x).
$$
 Let us define $\tilde T: X^{\otimes 3}\longrightarrow X$ by the assignment 
$$
\tilde T((a,x)\otimes (b,y)\otimes (c,z)) = (abc, bcx - c[x,y] - b[x,z]+[[x,y]z]).
$$
\begin{lemma}
	The linear map $\tilde T$ turns $(X,\Delta)$ into a TSD object. 
\end{lemma}
\begin{proof}
	We need to show that $\tilde T$ is TSD. To do so, we follow the same procedure that shows that $T$ is TSD, by showing that $\tilde T$ can be written as a composition of binary self-distributive morphisms. On simple tensors we define the morphisms $\tilde q: X\otimes X\longrightarrow X$ by the assignment $(a,x)\otimes (b,y) \mapsto (ab, bx - [x,y])$. Let us now verify that $\tilde q$ satisfies (binary) self-distributivity. For the LHS of self-distributive condition we have on tensors $(a,x)\otimes (b,y)\otimes (c,z)\in X^{\otimes 3}$
	$$
	\tilde q(\tilde q \otimes \mathbbm 1)(a,x)\otimes (b,y)\otimes (c,z)  =  (abc,bcx - c[x,y] - b[x,z] + [[x,y],z]).
   $$
   For RHS of self-distributivity, we need to consider first the comultiplication $\Delta$, and transpose the middle terms. The computation explicitly gives
   \begin{eqnarray*}
   	\lefteqn{\tilde q\circ(\tilde q\otimes \tilde q)\circ(\mathbbm 1\otimes \tau\otimes \mathbbm 1)\circ(\mathbbm 1^{\otimes 2}\otimes \Delta)(a,x)\otimes (b,y)\otimes (c,z)}\\
   	&=& \tilde q(\tilde q(a,x)\otimes (c,z) \otimes \tilde q(b,y)\otimes (1,0) + \tilde q(a,x)\otimes (1,0) + \tilde q(b,y)\otimes (0,z)\\
   	&=&  (abc,bcx - b[x,z] - c[x,y] + [[x,z],y] + [x,[y,z]]).
   	\end{eqnarray*}
   Using the Jacobi identity for $L$ we see that the two terms coincide, showing that $\tilde q$ satisfies the (binary) self-distributive condition. Observe now that $\tilde T$ is obtained from $\tilde q$ by nesting a copy of $\tilde q$ into another copy of $\tilde q$: $\tilde T = \tilde q\circ (\tilde q\otimes \mathbbm 1)$. Applying Theorem~8.6 in \cite{ESZ} it follows that $\tilde T$ is TSD. To complete the proof we need to show that $\tilde T$ is a morphism of coalgebras, where $X^{\otimes 3}$ is endowed with the coalgebra structure induced by tensor product of coalgebras. This would follow if we show that $\tilde q$ is a morphism of coalgebras. We have
   \begin{eqnarray*}
   \Delta \tilde q (a,x)\otimes (b,y) &=& \Delta (ab,bx - [x,y])\\
   &=& (ab,bx - [x,y])\otimes (1,0) + (1,0)\otimes (0,bx - [x,y])
   \end{eqnarray*}
and also 
\begin{eqnarray*}
	\lefteqn{(\tilde q\otimes \tilde q) \circ (\mathbbm 1\otimes \tau \otimes \mathbbm 1)\circ (\Delta\otimes \Delta)(a,x)\otimes (b,y)}\\
	&=& (\tilde q\otimes \tilde q)((a,x)\otimes (b,y)\otimes (1,0)\otimes (1,0) + (a,x)\otimes  (1,0)\otimes (1,0)\otimes (0,y)\\
	&& + (1,0)\otimes (b,y)\otimes (0,x)\otimes (1,0) + (1,0)\otimes (1,0)\otimes (0,x)\otimes (0,y)\\
	&=& (ab,bx-[x,y])\otimes (1,0) + (a,x)\otimes (0,0)\\
	&& + (b,0)\otimes (0,x) + (1,0)\otimes (0,-[x,y])\\
	&=& (ab,bx - [x,y])\otimes (1,0) + (1,0)\otimes (0,bx - [x,y]).
	\end{eqnarray*}
    A direct computation shows also that $\tilde q$ respects the counit $\epsilon$. 
\end{proof}
\begin{remark}
	{\rm 
A direct computation to show that $\tilde T$ satisfies the TSD property can be done, with a similar approach as in Appendix~A in \cite{ESZ}, where the only difference will be in the signs appearing in the monomials, due to the fact that $\tilde q$ is defined with a ``twisted'' term  with respect to the analogous binary operation that gives rise to $T$. 
}
\end{remark}
\begin{proposition}\label{pro:reverse}
	The ternary operation $T$satisfies the reversibility condition with respect to $\tilde T$.
\end{proposition}
\begin{proof}
	We verify the reversibility condition on simple tensors directly. We compute 
	\begin{eqnarray*}
   \lefteqn{\tilde T(T((a,x)\otimes (b,y)^{(2)} \otimes (c,z)^{(2)})\otimes (c,z)^{(1)}\otimes (b,y)^{(1)})} \\
   &=& \tilde T(T((a,x)\otimes (1,0)\otimes (1,0))\otimes (c,z)\otimes (b,y)) \\
   && + \tilde T(T((a,x)\otimes (1,0)\otimes (0,z))\otimes (1,0)\otimes (b,y))\\
   && + \tilde T(T((a,x)\otimes (0,y)\otimes (1,0))\otimes (c,z)\otimes (1,0))\\
   && + \tilde T(T((a,x)\otimes (0,y)\otimes (0,z))\otimes (1,0)\otimes (1,0))\\
   &=& (abc,bcx - c[x,y] - b[x,z]+ [[x,z],y])\\
   && \tilde T((0,[x,z])\otimes (1,0)\otimes (b,y))\\
   && + \tilde T((0,[x,y])\otimes (c,z)\otimes (1,0))\\
   && + \tilde T((0,[[x,y],z])\otimes (1,0)\otimes (1,0))\\
   &=& (abc,bcx - c[x,y] - b[x,z]+ [[x,z],y])+ (0,b[x,z]-[[x,z],y]) \\
   && + (0, c[x,y]-[[x,y],z]) + (0,[[x,y],z])\\
   &=&  (abc, bc x)\\
   &=& bc (a,x)\\
   &=& \epsilon(b,y)\epsilon(c,z) (a,x),
	\end{eqnarray*}
which proves the reversibility condition. A similar computation holds for the case in which the roles of $\tilde T$ and $T$ are exchanged. 
\end{proof}
\subsection{Invertibility of the pre-Yang-Baxter operator $R$}

Applying the reversibility condition and Proposition~\ref{pro:reverse} we construct an inverse to $R$, showing that this is indeed a Yang-Baxter operator. Before the statement and proof of this assertion, we make an observation.
\begin{remark}
	{\rm 
	The proof of Theorem~\ref{thm:Rinv} below holds more generally when $X$ is a cocommutative Lie object in a monoidal category, and $R$ is constructed from a TSD object that is reversible. Proofs are generalized to this case as done in \cite{EZ} for quantum heap objects, Section~7, and the equations in the proof below are replaced by commutative diagrams. We will not pursue this theory in its full generality here. 
}
\end{remark}

\begin{theorem}\label{thm:Rinv}
	The operator $R$ defined in Subsection~\ref{subsec:braid} is invertible. Therefore it is a Yang-Baxter operator. 
\end{theorem}
\begin{proof}
	We define the map $\tilde R$ as
	\begin{eqnarray*}
		\lefteqn{(a,x)\otimes (b,y)\otimes (c,z) \otimes (d,w)}\\
		&\mapsto& \tilde T((c,z)\otimes (b,y)^{(2)}\otimes (a,x)^{(2)})\otimes \tilde T((d,w)\otimes (b,y)^{(3)}\otimes (a,x)^{(3)})\otimes\\
		&&\hspace{0.5cm} \otimes  (a,x)^{(1)}\otimes (b,y)^{(1)},
	\end{eqnarray*}
and show that it is the required inverse to $R$. This result, in fact, depends only on the fact that $T$ satisfies the reversibility condition with respect to $\tilde T$, as shown in Proposition~\ref{pro:reverse}, rather than on the actual definitions of $T$ and $\tilde T$. We have
\begin{eqnarray*}
\lefteqn{\tilde R R(a,x)\otimes (b,y)\otimes (c,z)\otimes (d,w)}\\
&=& \tilde R(c,z)^{(1)}\otimes (d,w)^{(1)}\otimes\\
&& T((a,x)\otimes (c,z)^{(2)}\otimes (d,w)^{(2)}) \otimes T((b,y)\otimes (c,z)^{(3)}\otimes (d,w)^{(3)}) \\
&=& \tilde T(T((a,x)\otimes (c,z)^{(2)}\otimes (d,w)^{(2)})\otimes (d,w)^{(12)}\otimes (c,z)^{(12)}\otimes \\
&&  \tilde T(T((b,y)\otimes (c,z)^{(3)}\otimes (d,w)^{(3)})\otimes (d,w)^{(13)}\otimes (c,z)^{(13)}\\
&& \otimes (c,z)^{(11)}\otimes (d,w)^{(11)}.
\end{eqnarray*}
Using naturality of the switching morphism and the fact that $T$ and $\tilde T$ are morphims of coalgebras, we can rewrite the last term as 
$$
(\tilde T\circ (T\otimes \mathbbm 1^{\otimes 2}))^{\otimes 2}\circ \shuffle \circ (\mathbbm 1^{\otimes 2}\otimes \Delta_4\otimes \Delta_4)(a,x)\otimes (b,y)\otimes (c,z)\otimes (d,w),
$$
where the permutation $\shuffle$ is given by 
$$
\shuffle = \bigl(\begin{smallmatrix}
1 & 2 & 3 & 4 & 5 &6 & 7 & 8 & 9 & 10 & 11 & 12\\
1 & 6 & 11 & 5 & 10 & 2 & 7 & 12 & 4 & 9 & 3 & 8
\end{smallmatrix}\bigr).
$$
Since $\Delta_4$ is symmetric with respect to permutations we can compose $\Delta_4\otimes \Delta_4$ with the permutation 
$$
\shuffle' = \bigl(\begin{smallmatrix}
1 & 2 & 3 & 4 & 5\\
1 & 2 & 4 & 3 & 5
\end{smallmatrix}\bigr) \otimes 
\bigl(\begin{smallmatrix}
1 & 2 & 3 & 4 & 5\\
1 & 2 & 4 & 3 & 5
\end{smallmatrix}\bigr)
$$
leaving the result invariant. Computing 
$$
(\tilde T\circ (T\otimes \mathbbm 1^{\otimes 2}))^{\otimes 2}\circ \shuffle \circ (\mathbbm 1^{\otimes 2}\otimes \Delta'_4\otimes \Delta'_4)(a,x)\otimes (b,y)\otimes (c,z)\otimes (d,w),
$$
where we have set $\Delta'_4 = \shuffle'\circ \Delta$, we obtain
\begin{eqnarray*}
\lefteqn{\tilde R R(a,x)\otimes (b,y)\otimes (c,z)\otimes (d,w)}\\
&=& \tilde T(T((a,x)\otimes (c,z)^{(3)}\otimes (d,w)^{(3)})\otimes (d,w)^{(2)}\otimes (c,z)^{(2)}) \\
&& \otimes \tilde T(T((b,y)\otimes (c,z)^{(5)}\otimes (d,w)^{(5)})\otimes (d,w)^{(4)}\otimes (c,z)^{(4)}\\
&& \otimes (c,z)^{(1)}\otimes (d,w)^{(1)}\\
&=& \epsilon(c,z)^{(2)}\epsilon(d,w)^{(2)} (a,x)\otimes (b,y) \otimes (c,z)^{(1)}\otimes (d,w)^{(2)}\\
&=& (a,x)\otimes (b,y) \otimes (c,z)\otimes (d,w),
\end{eqnarray*}
which shows that $\tilde R$ is a left inverse to $R$. Similar considerations apply to show that $\tilde R$ is a right inverse as well. 
\end{proof}

\section{Generalization to $3$-Lie algebras}\label{sec:ternaryLie}

\subsection{3-Lie algebras}
In this subsection we consider a generalization of the TSD structure defined by composing Lie algebras, to the case of $3$-Lie algebras and, in particular, to that of ternary Nambu-Lie algebras. See for instance the case $n=3$ in \cite{Filippov}. In fact, from a (binary) Lie algebra bracket $[\bullet,\bullet]$, the trilinear bracket defined by the rule $[\bullet,\bullet,\bullet] := [[\bullet,\bullet],\bullet]$ statisfies the defining Filippov Identity of a $3$-Lie algebra introduced in \cite{Filippov}. It is natural to ask whether we can construct TSD structures directly from a general $3$-Lie algebra, even when this is not obtained as the composition of binary structures. 

Let $L$ be a $3$-Lie algebra over a field $\mathbbm k$, where the trilinear bracket is denoted by $[\bullet, \bullet, \bullet]$. Define the vector space $X := \mathbbm k\oplus L$, where similarly as before we indicate its elements by pairs of a scalar and a vector in $L$. We endow $L$ with the comultiplication introduced in Section~\ref{sec:LieTSD}, and define the ternary operation $T: X\otimes X\otimes X\longrightarrow X$ on simple tensors according to the assignment
$$
(a,x)\otimes (b,y)\otimes (c,z) \mapsto (abc, bcx + [x,y,z]).
$$

\begin{lemma}\label{lem:3Lie}
	The coalgebra $(X,\Delta)$ endowed with the ternary operation $T$ defined above is a TSD object. Moreover, $T$ is reversible.
\end{lemma}
\begin{proof}
	First we need to prove that $T$ is a morphism of coalgebras. This means that the following diagram
	\begin{center}
		\begin{tikzcd}
			X\otimes X\otimes X\arrow[rr,"\shuffle \circ \Delta_3^{\otimes 3}"]\arrow[d,"T"] & & X^{\otimes 3}\otimes X^{\otimes 3}\otimes X^{\otimes 3}\arrow[d,"T^{\otimes 3}"]\\
			X\arrow[rr,"\Delta_3"] & & X\otimes X\otimes X
			\end{tikzcd}
	\end{center}
commutes, where $\shuffle = \bigl(\begin{smallmatrix}
1 & 2 & 3 & 4 & 5 &6 & 7 & 8 & 9 \\
1& 4 & 7 & 2 & 5 & 8 & 3 & 7 & 9
\end{smallmatrix}\bigr)$. 
\begin{figure}[h!]
	\begin{eqnarray*}
		\begin{tikzpicture}[scale = 0.7,baseline={([yshift=-0.1cm]current bounding box.center)},vertex/.style={anchor=base,
			circle,fill=black!25,minimum size=18pt,inner sep=2pt}]
		\draw (-3,2) -- (0,0);
		\draw (0,2) -- (0,0);
		\draw (3,2) -- (0,0);
		
		\draw (0,0) -- (0,-2);
		
		\draw (-3,-4) -- (0,-2);
		\draw (0,-4) -- (0,-2);
		\draw (3,-4) -- (0,-2);
		
		\node (a)  at (0.5,0) {$T$};
		\node (a)  at (0.5,-2) {$\Delta_3$};
		\end{tikzpicture}
		&=&
		\begin{tikzpicture}[scale=0.4,baseline={([yshift=-0.1cm]current bounding box.center)},vertex/.style={anchor=base,
			circle,fill=black!25,minimum size=18pt,inner sep=2pt}]
		\draw (-7,13) -- (-5,10);
		\draw (0,13) -- (0,10);
		\draw (7,13) -- (5,10);
		
		\draw (-5,10) -- (-5,5);
		\draw (-5,10) -- (0,5);
		\draw (-5,10) -- (5,5);
		
		\draw (0,10) -- (-5,5);
		\draw (0,10) -- (0,5);
		\draw (0,10) -- (5,5);
		
		\draw (5,10) -- (-5,5);
		\draw (5,10) -- (0,5);
		\draw (5,10) -- (5,5);
		
		\draw (-5,5) -- (-5,0);
		\draw (0,5) -- (0,0);
		\draw (5,5) -- (5,0);
		
		\node (a)  at (-5.7,10) {$\Delta_3$};
		\node (a)  at (0.8,10.2) {$\Delta_3$};
		\node (a)  at (5.7,10) {$\Delta_3$};
		
		\node (a)  at (-5.6,5) {$T$};
		\node (a)  at (0.7,5) {$T$};
		\node (a)  at (5.6,5) {$T$};
		\end{tikzpicture}
	\end{eqnarray*}
\caption{$T$ is a morphism of ternary coalgebras}
\label{fig:coalgebramprhism}
\end{figure}
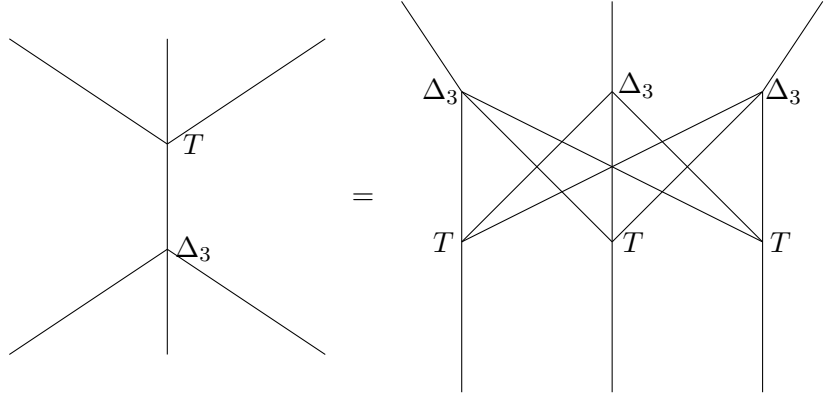
 This commutative diagram corresponds to the string diagram given in Figure~\ref{fig:coalgebramprhism}.
Since, in general, the ternary Lie bracket is not decomposable into the composition of two binary operations, as in Section~\ref{sec:LieTSD}, it is not possible to verify the commutativity of the diagram by verifying a binary intermediate step. It is useful to consider the following identities that follow directly from the definitions
\begin{eqnarray}
 T((a,x)\otimes (1,0)\otimes (1,0)) &=& (a,x)\label{eqn:1}\\
T((a,x)\otimes (b,0)\otimes (0,z)) &=& (0,0)\\
T((a,x)\otimes (0,y)\otimes (c,z)) &=& T((a,x)\otimes (b,y)\otimes (0,z)) = (0,[x,y,z])\\
T((1,0)\otimes (b,y)\otimes (c,z)) &=& (bc,0).\label{eqn:4}
\end{eqnarray}
The lower half of the perimeter of the diagram, applied to the basis vector $(a,x)\otimes (b,y)\otimes (c,z)$, is seen to give
\begin{eqnarray*}
 \Delta_3\circ T((a,x)\otimes (b,y)\otimes (c,z)) &=& \Delta_3(abc,bcx+[x,y,z]) \\
 &=& (abc,bcx + [x,y,z])\otimes (1,0)\otimes (1,0)\\
 && + (1,0)\otimes (0,bcx + [x,y,z])\otimes (1,0)\\
 && + (1,0)\otimes (1,0)\otimes (0,bcx+[x,y,z]).
\end{eqnarray*}
The upper half of the perimeter is given by a sum of 27 terms, corresponding to the fact that it is the tensor product of the output of three copies of $\Delta_3$, each of which is given by the sum of three terms. We do not write all the terms, but applying Equations~(\ref{eqn:1})--(\ref{eqn:4}) the only ones that are nontrivial are those corresponding to the terms $111$, $211$, $222$, $311$ and $333$, where the triples of numbers indicate the lexicographical ordering of the 27 terms in the natural way. We obtain
\begin{eqnarray*}
\lefteqn{T^{\otimes 3}\shuffle\Delta_3^{\otimes 3}(a,x)\otimes (b,y)\otimes (c,z)} \\
&=& (abc,bcx+[x,y,z])\otimes (1,0)\otimes (1,0) + (bc,0)\otimes (0,x)\otimes (1,0) \\
&& + (1,0)\otimes (0,[x,y,z])\otimes (1,0) + (bc,0)\otimes (1,0)\otimes (0,x)\\
&& + (1,0)\otimes (1,0)\otimes (0,[x,y,z]).
\end{eqnarray*}
Grouping second and third summand, as well as fourth and fifth summand in the previous equation gives the required equality. So, $T$ is a coalgebra morphism. 

We need to prove now that $T$ satisfies the TSD property. Once again we consider a simple tensor of type $(a,x)\otimes (b,y)\otimes (c,z)\otimes (d,u)\otimes (e,v)$. The LHS of the TSD condition reads
\begin{eqnarray*}
\lefteqn {T(T((a,x)\otimes (b,y)\otimes (c,z))\otimes (d,u)\otimes (e,v))}\\
&=& (abcde, bcdex + de[x,y,z]+ bc [x,u,v] + [[x,y,z],u,v]),
\end{eqnarray*}
while the RHS is a sum of 9 terms, 6 of which are trivial. The nontrivial terms give
\begin{eqnarray*}
\lefteqn{T(T\otimes T\otimes T)\shuffle_3 (\mathbbm 1^{\otimes 3}\otimes \Delta_3^{\otimes 2})(a,x)\otimes (b,y)\otimes (c,z)\otimes (d,u)\otimes (e,v)}\\
&=& (abcde, bcdex + de[x,y,z] + bc[x,u,v] + [[x,u,v],y,z])\\
&& + (0,[x,[y,u,v],z]) + (0,[x,y,[z,u,v]]).
\end{eqnarray*}
It follows that LHS and RHS differ by one application of the Filippov Identity for $n=3$, and $T$ satisfies the TSD property. 

Lastly, we show that $T\circ (\mathbbm 1\otimes \tau)$ is the required inverse of $T$, where $\tau$ is the canonical switching map on a tensor product of vector spaces. On simple tensors $(a,x)\otimes (b,y)\otimes (c,z)$ we have
\begin{eqnarray*}
\lefteqn{T(T((a,x)\otimes (b,y)^{(1)}\otimes (c,z)^{(1)})\otimes (c,z)^{(2)}\otimes (b,y)^{(2)})}\\
&=& (abc,bcx+[x,y,z]) + T((ac,cx)\otimes (1,0)\otimes (0,y))\\
&& + T((ab,bx)\otimes (0,z)\otimes (1,0)) + T((a,x)\otimes (0,z)\otimes (0,y)) \\
&=& (abc,bcx+[x,y,z]) + (0,[x,z,y])\\
&=& (abc,bcx) \\
&=& \epsilon(b,y) \epsilon (c,z) \cdot (a,x).
\end{eqnarray*}
The proof of the lemma is complete.
\end{proof}

Let $L$ be a $3$-Lie algebra and let $X = \mathbbm k \oplus L$ with the ternary operation $T$ defined above. Define the operator $R: X^{\otimes 2}\otimes X^{\otimes 2}\longrightarrow X^{\otimes 2}\otimes X^{\otimes 2}$ on basis vectors as
\begin{eqnarray*}
(a,x)\otimes (b,y) \otimes (c,z) \otimes (d,w) &\mapsto&
 (c,z)^{(1)}\otimes (d,w)^{(1)} \\
 &&\otimes T((a,x)\otimes (c,z)^{(2)}\otimes (d,w)^{(2)})\\
 && \otimes T((b,y)\otimes (c,z)^{(3)}\otimes (d,w)^{(3)}).
\end{eqnarray*}

Observe that this is, formally, the same definition as for the case obtained by composing binary SD operations. Then we have the following result. 

\begin{corollary}
	The operator $R: X^{\otimes 2}\otimes X^{\otimes 2}\longrightarrow X^{\otimes 2}\otimes X^{\otimes 2}$ is a Yang-Baxter operator. 
\end{corollary}
\begin{proof}
	First, we observe that the proofs of Theorem~\ref{thm:YBTSD} and Theorem~\ref{thm:Rinv} do not depend on the specific form of the TSD operation $T$, but rather on the fact that $T$ is a reversible TSD operation, and the formal definition of the Yang-Baxter operator $R$ from $T$. Therefore, applying Lemma~\ref{lem:3Lie} we can follow the same procedure and complete the proof, since the operator at hand is defined in the same way as in the case of the operator of Theorem~\ref{thm:YBTSD} and Theorem~\ref{thm:Rinv}.
\end{proof}

\section{Framed knot and link invariants}\label{sec:invariants}

It is known that given a YB operator $R$ on a vector space $V$, it is possible to associate representations of the braid group, as well as link invariants to it \cite{Tur}. In this section we show that the YB operators $R: X^{\otimes 2} \otimes X^{\otimes 2} \longrightarrow X^{\otimes 2}\otimes X^{\otimes 2}$ associated to the Lie algebras of Section~\ref{sec:LieTSD} and Section~\ref{sec:ternaryLie} can be applied to define representations of the framed braid group and, consequently, invariants of framed links. To this purpose, we need to define an invertible twisting map $X^{\otimes 2} \longrightarrow X^{\otimes 2}$ that commutes with braidings induced by the YB operator $R$. 

Let $(X,T)$ be a TSD object as in Section~\ref{sec:LieTSD} or \ref{sec:ternaryLie}. We define the twist map $\theta : X\otimes X \longrightarrow X\otimes X$ by extending the following assignment by linearity
\begin{eqnarray*}
\lefteqn{(a,x)\otimes (b,y) \mapsto}\\
&& T((a,x)^{(1)}\otimes (a,x)^{(2)}\otimes (b,y)^{(2)}) \otimes T((b,y)^{(1)}\otimes (a,x)^{(3)}\otimes (b,y)^{(3)}).
\end{eqnarray*}
The geometric interpretation of this operation is that of a ribbon self-crossing, which introduces a full twist, as it is shown in Figure~\ref{fig:slicedtwist}. Here, the diagrams are sliced horizontally in order to better depict the successive decomposition of the maps that give rise to the twist.
 	\begin{figure}
 		\begin{center}
 			\begin{tikzpicture}
 			\draw  (0,0) arc (0:180:1cm and 1cm); 
 			\draw  (-0.3,0) arc (0:180:0.7cm and 0.7cm); 
 			
 			\draw (1,2) -- (1,0);
 			\draw (1.3,2) -- (1.3,0);
 			
 			\draw[dashed] (-2,-0.5) -- (1.5,-0.5);
 			
 			\draw (-2,-1) -- (-2,-3);
 			\draw (-1.7,-1) -- (-1.7,-3);
 			
 			\draw (1,-1) -- (0,-3);
 			\draw (1.3,-1) -- (0.3, -3);
 			
 			\draw (0,-1) -- (0.4,-1.8);
 			\draw (0.3,-1) -- (0.56,-1.55);
 			
 			\draw (0.72,-2.32) -- (1,-3); 
 			\draw (0.92,-2.15) -- (1.3,-3);
 			
 			\draw[dashed] (-2,-3.5) -- (1.5,-3.5);
 			
 			\draw  (-2,-4) arc (180:360:1cm and 1cm); 
 			\draw  (-1.7,-4) arc (180:360:0.7cm and 0.7cm); 
 			
 			\draw (1,-4) -- (1,-6);
 			\draw (1.3,-4) -- (1.3,-6);
 			\end{tikzpicture}
 		\end{center}
 	\caption{Geometric definition of the twist map, where the horizontal slicing corresponds to the decomposition into creation/annihilation maps and $R$}
 	\label{fig:slicedtwist}
 	\end{figure}
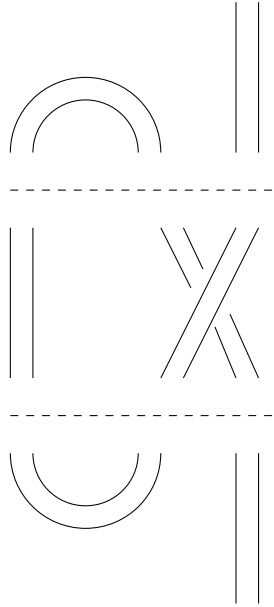
 	In fact, let $\{w_i\}$ denote a basis of $X$, e.g. $\{(1,0),(0,x_1), \ldots , (0,x_n)\}$, and let $w_i^*$ denote the functional $w_i^*: X\longrightarrow \mathbbm k$ dual to $w_i$. We have natural pairing and copairing maps on $X^{\otimes 2}$, denoted by the symbols $\doublecup$ and $\doublecap$ respectively, given by the equations $\doublecup(w_k^*\otimes w_i^*\otimes w_j\otimes w_\ell) = \delta_{ij}\delta_{k\ell}$ and $\doublecap(1) = \sum_{i,j} w_i^*\otimes w_j^*\otimes w_j\otimes w_i$.  Let us indicate basis vectors in $X$ as $u$ and $v$ for simplicity.
Then, on simple tensors, we have
\begin{eqnarray*}
u\otimes v &\mapsto& w_i^*\otimes w_j^*\otimes w_j\otimes w_i\otimes u\otimes v\\
&\mapsto& w_i^*\otimes w_j^*\otimes u^{(1)}\otimes v^{(1)}\otimes T(w_j\otimes u^{(2)}\otimes v^{(2)})\otimes  T(w_i\otimes u^{(3)}\otimes v^{(3)})\\
&\mapsto& T(u^{(1)}\otimes u^{(2)}\otimes v^{(2)})\otimes  T(v^{(1)}\otimes u^{(3)}\otimes v^{(3)}),
\end{eqnarray*}
where each step corresponds to the slicing of Figure~\ref{fig:slicedtwist}. We note that although this geometric interpretation requires some regularity condition on $X$, i.e. existence of dual basis as it is enforced for instance by having a finitely generated projective module over $\mathbbm k$, the definition of $\theta$ itself does not make use of any assumption and it is applicable to any Lie algebra over a unital ground ring $\mathbbm k$. 

We want to argue now that $\theta$ can be slid beneath and above crossings, in the sense that the following holds.
\begin{lemma}
	Let $(X,T)$ denote the TSD object arising from a binary Lie algebra or a $3$-Lie algebra, as in Sections~\ref{sec:LieTSD} and~\ref{sec:ternaryLie}. Then we have
	\begin{eqnarray}
	R \circ (\theta\otimes \mathbbm 1) &=& (\mathbbm 1\otimes \theta) \circ R\\ \label{eqn:leftsliding}
	R \circ (\mathbbm 1\otimes \theta) &=& (\theta\otimes \mathbbm 1)\circ R.
	\end{eqnarray} 
\end{lemma}
\begin{proof}
 It is seen, by direct computation, that when applied to a basis vector $(a,x)\otimes (b,y)$ the LHS and RHS of Equation~\ref{eqn:leftsliding} differ by an application of the TSD property of $T$, upon rearranging the terms of coproduct that are not in the first entry, and applying the fact that $T$ is a coalgebra morphism. 
\end{proof}

Let us define the map $\Phi_n : \mathbb {FB}_n \longrightarrow {\rm Aut}(X^{\otimes 2n})$ from the braid group to the automorphism group of $X^{\otimes 2n}$ for each $n$, where $X$ is the TSD object defined from a binary or ternary Lie algebra as above. We send the generators $\sigma_i$ to the automorphism $\mathbbm 1^{\otimes (2i-2)}\otimes R \otimes \mathbbm 1^{\otimes (2n-2i+2)}$, and the generator $t_i$, corresponding to the $n$-tuple $(0, \ldots, 1, \ldots, 0)$ when using the definition of $\mathbb {FB}_n $ as semi-direct product, to $\mathbbm 1^{\otimes (2i-2)}\otimes \theta\otimes \mathbbm 1^{\otimes (2n-2i)}$, where $1$ is assumed to be in the $i^{\rm th}$ position. We obtain therefore a map $\Phi_{\infty} : \bigcup_n \mathbbm {FB}_n \longrightarrow \bigcup_n {\rm Aut}(X^{\otimes 2n})$, where $\mathbbm {FB}_\infty := \bigcup_n \mathbbm {FB}_n$ is the infinite framed braid group. The symbol $\bigcup_n$ here is an improper way to indicate the colimit of groups, where the injection $\mathbb{FB}_n \hookrightarrow \mathbb{FB}_{n+1}$ is the natural one, and ${\rm Aut}(X^{\otimes 2n}) \hookrightarrow {\rm Aut}(X^{\otimes 2(n+1)})$ is obtained by tensoring (on the right) with the identity map $\mathbbm 1^{\otimes 2}$. 

We need a preliminary lemma, before giving the main result of this section. 

\begin{lemma}\label{lem:mixed}
	Let $(X,T)$ be a ternary rack object in the category of vector spaces, then $T$ and its inverse $\tilde T$ satisfy a ``mixed'' distributivity condition
	\begin{eqnarray*}
	\lefteqn{T(\tilde T(x\otimes y\otimes z)\otimes u\otimes v)} \\
	&=& \tilde T(T(x\otimes u^{(1)}\otimes v^{(1)})\otimes T(y\otimes u^{(2)}\otimes v^{(2)}) \otimes T(z\otimes u^{(3)}\otimes v^{(3)})),
	\end{eqnarray*}
for each $x,y,z,u,v\in X$, and a similar equation holds exchanging the roles of $T$ and $\tilde T$. 
\end{lemma}
\begin{proof}
	We have the following chain of equalities
	\begin{eqnarray*}
	\lefteqn{T(\tilde T(x\otimes y\otimes z)\otimes u\otimes v)}\\
	 &=& T(\tilde T(\tilde T(T(x\otimes u^{(2)}\otimes v^{(2)})\otimes v^{(3)}\otimes u^{(3)})\\
	 &&\otimes \tilde T(T(y\otimes u^{(4)}\otimes v^{(4)})\otimes v^{(5)}\otimes u^{(5)})\\
	 &&\otimes \tilde T(T(z\otimes u^{(6)}\otimes v^{(6)})\otimes v^{(7)}\otimes u^{(7)}))\\
	 &&\otimes u^{(1)}\otimes v^{(1)})\\
	 &=& T(\tilde T(\tilde T(T(x\otimes u^{(5)}\otimes v^{(5)})\otimes v^{(2)}\otimes u^{(2)})\\
	 &&\otimes \tilde T(T(y\otimes u^{(6)}\otimes v^{(6)})\otimes v^{(3)}\otimes u^{(3)})\\
	 &&\otimes \tilde T(T(z\otimes u^{(7)}\otimes v^{(7)})\otimes v^{(4)}\otimes u^{(4)}))\\
	 &&\otimes u^{(1)}\otimes v^{(1)})\\
	 &=& T(\tilde T(\tilde T(T(x\otimes u^{(3)}\otimes v^{(3)})\otimes T(y\otimes u^{(4)}\otimes v^{(4)})\\
	 &&\otimes T(z\otimes u^{(5)}\otimes v^{(5)})) \otimes v^{(2)}\otimes u^{(2)})\otimes u^{(1)}\otimes v^{(1)})\\
	 &=& \tilde T(T(x\otimes u^{(1)}\otimes v^{(1)})\otimes T(y\otimes u^{(2)}\otimes v^{(2)})\otimes T(z\otimes u^{(3)}\otimes v^{(3)})),
	\end{eqnarray*}
where the first equality uses the reversibility of $T$ with respect to $\tilde T$ thrice, the second equality is obtained by rearranging terms corresponding to the output of comultiplications, the third equality utilizes the TSD condition for $\tilde T$ and the last equality is obtained by another application of the reversibility condition of $T$ with respect to $\tilde T$. 
\end{proof}
\begin{remark}
	{\rm 
	This is a ternary analogue (in the category of vector spaces) of the known result for binary racks where the two operations $\lt$ and $\lt^{-1}$ satisfy an equation similar to the self-distributive property combined. See \cite{AC}, Lemma~70 for a proof in the binary and set-theoretic case.
}
\end{remark}

\begin{theorem}\label{thm:representation}
	The map $\Phi_\infty$ is a representation of the infinite framed braid group.
\end{theorem}
\begin{proof}
	The only part of the theorem which does not follow from the results already obtained is the fact that the twist is invertible, so that the algebraic counterpart of diagrammatic cancellation of kinks holds. We define the inverse of $\theta$ by means of reversibility property of $T$, according to the assignment 
	\begin{eqnarray*}
		\theta^{-1} (x\otimes y) = \tilde T(x^{(1)}\otimes y^{(2)}\otimes x^{(2)})\otimes \tilde T(y^{(1)}\otimes y^{(3)}\otimes x^{(3)}).
	\end{eqnarray*}
The proof that $\theta^{-1}$ is the inverse of $\theta$ is very similar to the proof that $\tilde R$ is the inverse of $R$. The main difference is that we need to use the TSD property where $T$ and $\tilde T$ appear both in the same equation, as given in Lemma~\ref{lem:mixed}. One also needs to reorder the elements appropriately.
\end{proof}

Given a $2$ or $3$-Lie algebra $L$ of finite dimension, we construct an invariant of framed links as follows. For a given framed link $\mathcal L$ we choose a framed braid, following the notation of \cite{KS}, $b = (t^1\cdots t^n) \tau_n$ where $\tau_n\in \mathbb B_n$ is a braid in the braid group $\mathbb B_n$ on $n$ strings, such that the closure of $b$ is isotopic to $\mathcal L$. Define maps $R_{n,i} : X^{\otimes 2n} \longrightarrow X^{\otimes 2n}$ as the tensor product $R_{n,i} := \mathbbm 1^{\otimes 2(i-1)} \otimes R \otimes \mathbbm 1^{\otimes 2 (n-i)}$. Let us now write $\tau_n$ as a product of braid group generators $\sigma_i$, say $b = \prod_j \sigma_{i_j}^{k_j}$. Then we define the map $\Psi_b: X^{\otimes 2n} \longrightarrow X^{\otimes 2n}$ as $\Psi_b = \prod_j R_{n,i_j}^{k_j} (\theta^{t_1}\otimes \cdots \otimes \theta^{t_n})$, where the product here indicates composition of maps. The trace of the map $\Psi_b$ is denoted by $\Psi(\mathcal L)$. 

\begin{remark}
	{\rm
	We observe that it is now needed a finiteness condition on the dimension of $L$, since the trace of the operator $\Psi(\mathcal L)$ is not necessarily finite otherwise. Up to now, the constructions of this article did not depend on having an infinite or finite dimensional Lie algebra. 
}
\end{remark}

The reason to consider the trace of the operators described above is contained in the following. 
\begin{corollary}\label{cor:invariant}
	{\rm
	Let $L$ be an $n$-Lie algebra over $\mathbbm k$, where $n= 2,3$, and let $\mathcal L$ be a framed link which is the closure of a framed braid $b$. Then the map $\Psi_b$ depends only on the isotopy class of $\mathcal L$ and, therefore, $\Psi(\mathcal L)$ is a $\mathbbm k$-valued invariant of $\mathcal L$.  
}
\end{corollary}
\begin{proof}
	
	In fact, this follows from the fact that for a chosen framed $n$-braid $b\in \mathbb {FB}_n$ whose closure is $\mathcal L$, $\Psi(\mathcal L)$ is the trace of $\Phi_n (b)\in Aut(X^{\otimes (2n)})$, since by definition $\Psi_b = \Phi_n(b)$. In fact, from Theorem~\ref{thm:representation} $\Psi(\mathcal L)$ is independent on the choice of $b$ among equivalent framed braids. One then has to verify the framed version of Markov's theorem, see \cite{KS} Lemma~1. But, conjugation does not change the trace and, moreover, Markov's stabilization is used to introduce framings according to the definition of twist $\theta$ given above. 
	
	\end{proof}

\begin{remark}
	{\rm 
We observe that the result obtained in Corollary~\ref{cor:invariant} would still hold true if we used a Lie algebra which is finitely generated projective module over a unitary ring $\mathbbm k$, rather than a finite dimensional vector space. In fact, as previously pointed out, the original work of Filippov introduced $n$-Lie algebras over unitary rings, rather than vector spaces. Although the main structure theorems for $n$-Lie algebras are known in the stronger case of vector spaces, there is an abundance of examples of $n$-Lie algebras over unitary rings that are not fields. 
}
\end{remark}

\section{ Examples, Computations and Deformations} 

In this section we provide some considerations on the Yang-Baxter operators constructed in  Section~\ref{sec:invariants}, and give computations for some low-dimensional Lie algebras. Moreover, we determine the twisting maps $\theta^n$ for arbitrary binary and ternary Lie algebras, and see that the former give rise to nontrivial twists, while the latter just give trivial ones. As a consequence, while ternary operations obtained by composing binary Lie algebras provide representations of the framed braid group with nontrivial twists, $3$-Lie algebras give representations of the (unframed) braid group. It is not yet clear whether  the Yang-Baxter operators $R$ obtained produce new invariants. Due to the presence of terms of type $(1,0)$ in the comultiplication used to construct $R$, it seems possible that the corresponding invariants would coincide with those given by the switching map $\tau: X\otimes X \longrightarrow X\otimes X$, although this is not yet settled. Nonetheless, it turns out that $R$ is generally not {\it gauge equivalent} to $\tau$. The possibility that the link invariants coincide with those obtained using the switching map, prompts us to consider deformations of the operators obtained so far in the next subsection. 

For $L$ a Lie algebra or $3$-Lie algebra, and $X = \mathbbm k \oplus L$ the TSD object constructed in Section~\ref{sec:LieTSD} and Section~\ref{sec:ternaryLie} respectively, we have that a basis $\{x_1,\ldots, x_n\}$ of $L$ determines a basis of $X$ as $\{(1,0), (0,x_1), \ldots , (0,x_n)\}$. We will use this notation for the computations in the rest of this section. 

We recall the following definition, see \cite{Eis}.  
\begin{definition}
	{\rm 
	Two Yang-Baxter operators $R_1, R_2: V\otimes V \longrightarrow V\otimes V$ are said to be gauge equivalent if there exists an invertible map $g: V\longrightarrow V$ such that $R_2 = (g\otimes g)^{-1}\circ R_1\circ(g\otimes g)$. 
}
\end{definition}
\begin{remark}
	{\rm 
Yang-Baxter operators are classified up to gauge equivalence. Moreover, we observe that the matrices of two gauge equivalent operators are similar, although similarity is in general a weaker condition than gauge equivalence. 
}
\end{remark}

Let $\frak g_2$ denote the only nontrivial $2$-dimensional Lie algebra, i.e. with bracket determined by $[e_1,e_2] = e_1$, where $e_i$ are the basis vectors of $\frak g_2$. We have the following.

\begin{proposition}\label{pro:gauge}
	Let $R$ denote the Yang-Baxter operator corresponding to $\frak g_2$, and let $\tau$ denote the switching map on $X = \mathbb C\oplus \frak g_2$. Then, $R$ and $\tau$ are not gauge equivalent. 
\end{proposition}
\begin{proof}
	A computer aided computation shows that the Jordan normal form of $R$ contains $2$-dimensional Jordan blocks, whereas $\tau$ has a diagonal Jordan normal form, i.e. it consists of $1$-dimensional Jordan blocks. Since it is known that two matrices over $\mathbb C$ are similar if and only if they have the same Jordan normal form, it follows that $R$ and $\tau$ are not similar, and in particular they are not gauge equivalent. 
\end{proof}

\begin{remark}
	{\rm 
	From Proposition~\ref{pro:gauge} we have that the Yang-Baxter operators constructed in this article can be algebraically different from the switching maps, which are a well known class of nontrivial solutions to the Yang-Baxter equations. 
}
\end{remark}


Before concluding this subsection, we consider the twisting map $\theta$ in the case when $T$ is obtained from compositions of binary brackets, and when $T$ is derived from a ternary Lie bracket. 

\begin{proposition}\label{pro:binarytwist}
	Let $L$ be a Lie algebra, and let $T$ denote the associated TSD operation. Then, if $\theta$ is the twisting map obtained from $T$ as in Section~\ref{sec:invariants}, we have 
	$$
	\theta^n = \mathbbm 1 + n h,
	$$
	where $h((a,x)\otimes (b,y)) = (0,[x,y])\otimes (1,0) + (1,0)\otimes (0,[y,x])$, for all $n\in \mathbbm Z$.
\end{proposition}
\begin{proof}
	By direct computation we see that $\theta((a,x)\otimes (b,y)) = (a,x)\otimes (b,y) + (0,[x,y])\otimes (1,0) + (1,0)\otimes (0,[y,x])$. Applying $\theta$ $n$ times then gives the result when $n$ is non-negative. For negative $n$, one proceeds in the same way, with the only difference that negative signs appear due to the definition of $\tilde T$ and $\theta^{-1}$. Specifically we have $\theta^{-1}((a,x)\otimes (b,y)) = (a,x)\otimes (b,y) - (0,[x,y])\otimes (1,0) - (1,0)\otimes (0,[y,x])$.
\end{proof}

\begin{remark}
	{\rm 
It follows from Proposition~\ref{pro:binarytwist} that the representation of Theorem~\ref{thm:representation} assigns nontrivial twisting maps to the generators $t_i$ of the framed braid group.
}
\end{remark}

\begin{proposition}
	For a $3$-Lie algebra $\mathfrak g$ and $X$ as above, the representation in Theorem~\ref{thm:representation} factors through the infinite braid group $\mathbb B_\infty$. 
\end{proposition}
\begin{proof}
	Computing the twist map $\theta$ as in Proposition~\ref{pro:binarytwist}, we see that $\theta^n = \mathbbm 1$. Therefore the map $\Phi_{\infty}$ of Section~\ref{sec:invariants} maps the generators $(0,\ldots,0, 1,0, \ldots, 0)$ of the framed braid group to the unit of ${\rm Aut} (X^{2n})$ for all $n$. This implies that $\Phi_{\infty}$ factors through the infinite braid group $\mathbb B_\infty$.
\end{proof}

\subsection{Deformations} 

We now consider the problem of deforming the Yang-Baxter operators constructed above, in order to obtain new solutions to the Yang-Baxter equation. Our main objective is to have a method to derive new representations of the (framed) braid group, and possibly deformed invariants of links. We do not currently know to what extent deforming the Yang-Baxter operators produces deformed invariants, and a computational investigation is needed to this perspective. We defer the proofs of the facts used in this subsection to the Appendix, in order to provide a smoother exposition. We mention that the present deformation approach is part of a much more general theory that uses the TSD cohomology of TSD structures, in the sense of Section~7 of \cite{EZ}, and will appear elsewhere in its totality. This is based on known methods of Gerstenhaber \cite{Ger} that are also known to work for Lie algebras \cite{NR}. The resulting deformation of Yang-Baxter operators has a parallelism with \cite{Eis}. We present such results for compositions of binary Lie algebras, for simplicity, although the basic idea is applicable in general with only few formal changes. 

To this end, let $X$ denote a TSD object as above, where $L$ is a (binary) Lie algebra over $\mathbbm k$. Then, let $\mathbbm k[[\hbar]]$ denote the ring of power series over the formal variable $\hbar$, and coefficients in $\mathbbm k$. Here, $\hbar$ has the meaning of a deformation parameter and, since we are interested only in ``infinitesimal'' deformations at first, we quotient out $\mathbbm k[[\hbar]]$ by the ideal $(\hbar^2)$, therefore discarding deformations of higher order. Let us set $\mathbbm k' := \mathbbm k[[\hbar]]/(\hbar^2)$. We now extend the coefficients of $X$ to $\mathbbm k'$, and denote the resulting module by $X'$. Our objective now is that of finding $\Delta'$ and $T'$ on $X'$ that coincide with $\Delta$ and $T$ when restricting coefficients to $\mathbbm k$. We consider maps of the form $\Delta' = \Delta + \hbar \psi$ for some $\mathbbm k$-linear $\psi: X \longrightarrow X\otimes X$, for the comultiplication. Since $T = q\circ (q\otimes \mathbbm 1)$, i.e. it is obtained by composing binary operations, we consider deformations of $q$ as $q' = q + \hbar \phi$ for some $\phi: X\otimes X \longrightarrow X$, and then set $T' = q'\circ (q'\otimes \mathbbm 1)$. 

The final objective of the deformations is that of obtaining a deformed Yang-Baxter operator with respect to the one corresponding to $\Delta$ and $T$, according to the constructions found in Section~\ref{sec:LieTSD} and Section~\ref{sec:ternaryLie}. Constraining $\Delta'$ and $T'$ to satisfy the same compatibility conditions that $\Delta$ and $T$ satisfied, the new Yang-Baxter operator will automatically be a deformation of $R$ corresopnding to $\Delta$ and $T$. It is our objetive now to show the constraints on $\phi$ and $\psi$ needed to achieve this. 

\begin{lemma}\label{lem:deformedcoass}
	Let $(X,\Delta, \epsilon)$ be a coalgebra over $\mathbbm k$, and let $\Delta'$ be a deformation over $\mathbbm k'$ as above. Then, $\Delta'$ is coassociative if and only if 
	$$
	(\mathbbm 1\otimes \Delta)\psi + (\mathbbm 1\otimes \psi)\Delta = (\Delta\otimes \mathbbm 1)\psi + (\psi\otimes \mathbbm 1)\Delta,
	$$
	and
	$$
	(\epsilon\otimes \mathbbm 1)\psi = (\mathbbm 1\otimes \epsilon)\psi = 0.
	$$
\end{lemma}

\begin{remark}
	{\rm 
		We note that we are not deforming the counit $\epsilon$, although this is generally possible. We will not consider this possibility here, but we mention that allowing this would change the second group of equations in Lemma~\ref{lem:deformedcoass} and could provide a wider class of solutions for $\psi$. 
}
\end{remark}

\begin{example}\label{ex:deformedcoass}
	{\rm 
  With the comultiplication $\Delta$ considered in this article, thediagonal map $\psi : X \longrightarrow X\otimes X$ defined by $\psi(a,x) = (0,x)\otimes (0,x)$ satisfies the equations of Lemma~\ref{lem:deformedcoass}. So, the map $\Delta'(a,x) = (a,x)\otimes (1,0) + (1,0)\otimes (0,x) + \hbar (0,x)\otimes (0,x)$ is coassociative. The counit $\epsilon$ is extended to coefficients $\mathbbm k'$ trivially and it satisfies the counit axioms with respect to $\Delta'$. 
}
\end{example}

\begin{lemma}\label{lem:deformedSD}
	Let $(X,\Delta,q)$ be a binary self-distributive object, and let $\Delta'$ be such that $(X',\Delta')$ is a coalgebra. Let $\shuffle$ indicate the shuffle map corresponding to binary self-distributivity.  Then $q' = q + \hbar \phi$ satisfies the self-distributive condition with respect to $(X',\Delta')$ if and only if 
	\begin{eqnarray*}
	\lefteqn{q(\phi\otimes \mathbbm 1) + \phi(q\otimes \mathbbm 1)}\\
	&=& q[\phi\otimes q+q\otimes \phi]\shuffle (\mathbbm 1\otimes \Delta) + \phi(q\otimes q)\shuffle (\mathbbm 1\otimes \Delta) + q(q\otimes q)\shuffle (\mathbbm 1\otimes \psi).
	\end{eqnarray*}
Moreover, $q'$ is a morphism of coaglebras if and only if 
\begin{eqnarray*}
\Delta \phi + \psi q
&=& [\phi\otimes q + q\otimes \phi](\mathbbm 1\otimes \tau \otimes \mathbbm 1)(\Delta\otimes \Delta)\\
&& + (q\otimes q)(\mathbbm 1\otimes \tau \otimes \mathbbm 1)(\psi\otimes \Delta + \Delta\otimes \psi),
\end{eqnarray*}
where $\tau$ indicates the switching map.
\end{lemma}

\begin{example}\label{ex:deformedq}
	{\rm 
	Let $L$ be a Lie algebra. Let $X = \mathbbm k\oplus L$ denote the TSD object associated to $L$ as above. The map $\phi: X\otimes X\longrightarrow X$ defined on simple tensors as $\phi((a,x)\otimes (b,y)) = (0,[x,y])$ satisfies the equations of Lemma~\ref{lem:deformedSD} with respect to the deformation $\psi$ of $\Delta$ found in Example~\ref{ex:deformedcoass}. Therefore $q' = q + \hbar \phi$ is a deformation of $q$ with respect to the comultiplication $\Delta$. More generally, if $\sigma$ denotes any Lie algebra $2$-cocycle of $L$ with coefficients in $L$ itself, then $q' = q + \hbar \phi$, with $\phi((0,x)\otimes (0,y)) = (0,\sigma(x, y))$, is a deformation of $q'$. As it will be explained in more detail below, this example is part of a more general theory, where the deformations of an $n$-ary self-distributive operation are classified by a self-distributiive cohomology theory analogous to the one appeared in Section~7 of \cite{EZ}, and there is a map from Lie algebra cohomology to self-distributive cohomology. We will not consider the general case in the present article. 
}
\end{example}

We are interested in deforming the comultiplication $\Delta$, as well as the operation $q$.


\begin{example}\label{ex:deformed2d}
	{\rm 
		Let $\frak g_2$ denote the $2$-dimensional Lie algebra considered above, with coefficients in a field (or ring) $\mathbbm k$ of characteristic $2$. Let us consider a deformed $q' = q + \hbar \phi$ as in Example~\ref{ex:deformedq}. Let $\psi : X \longrightarrow \frak g_2\otimes \frak g_2\subset X\otimes X$ be defined on basis vectors as $\psi((1,0)) = 0$, $\psi((0,e_1)) = \alpha (0,e_1)\otimes (0,e_2) + \alpha(0,e_2)\otimes (0,e_1)$ and $\psi((0,e_2)) =\beta(0,e_1)\otimes (0,e_1) + \alpha(0,e_2)\otimes (0,e_2)$. Then, deform $\Delta$ as $\Delta' = \Delta + \hbar \psi$. By direct computation one sees that $\Delta'$ is coassociative, i.e. Lemma~\ref{lem:deformedcoass} holds, and moreover the equations in Lemma~\ref{lem:deformedSD} hold as well. Then, $(X,\Delta', q')$ is a self-distributive object. We defer the computations to the Appendix. This is the only deformation of $\Delta$ that is compatible with the $q'$ defined using the Lie bracket. 
	}
\end{example}

\begin{remark}\label{rmk:deformation}
	{\rm 
	The construction of Example~\ref{ex:deformed2d} can be applied more generally to solvable Lie algebras. In fact, assume that $L$ is a solvable Lie algebra such that $[L_r,L_r] = 0$, where $L_r$ denotes the $r^{\rm th}$-term in the derived series of $L$. Define $\psi : X \longrightarrow L_r\otimes L_r \subset X\otimes X$ such that the second equation in Lemma~\ref{lem:deformedSD} holds. Then, the deformed structure is a self-distributive object. This is a useful method to attempt deforming $(X,\Delta,q)$ for some special classes of Lie algebras.  
}
\end{remark}
We apply the same paradigm in the following.

\begin{example}\label{ex:Heisenberg}
	{\rm
		
	Let $\mathcal H$ denote the Heisenberg Lie algebra on $3$ generators $e_1, e_2, e_3$ with bracket defined by $[e_1,e_2] = e_3$, and zero otherwise. Let $X = \mathbb C \oplus \mathcal H$ denote the SD object associated to $\mathcal H$. Let $q' = q + \hbar \phi$ where $\phi(a,x)\otimes (b,y) = (0,[x,y])$ and $\Delta' = \Delta = \hbar \psi $ where $\psi(1,0) = \psi (0,e_3) = 0$, $\psi(0,e_1) =  \alpha_1^{23}[(0,e_2)\otimes (0,e_3) - (0,e_3)\otimes (0,e_2)] + \alpha_1^{33}(0,e_3)\otimes (0,e_3)$, and $\psi(0,e_2) = \alpha_2^{13}[(0,e_1)\otimes (0,e_3) - (0,e_3)\otimes (0,e_1)]  + \alpha_2^{33} (0,e_3)\otimes (0,e_3)$, for arbitrary coefficients $\alpha_i^{jk}\in \mathbb C$. Then, a (tedious) direct computation analogous to the one for Example~\ref{ex:deformed2d} shows that the equations of Lemma~\ref{lem:deformedcoass} and Lemma~\ref{lem:deformedSD} are satisfied. Therefore, we obtained a deformed SD object.
}
\end{example}

The same theory developed for $(X,\Delta,q)$ can be applied to the deformations $(X,\Delta',q')$, and we can therefore obtain ternary operations and Yang-Baxter operators, along with their corresponding representations of the (framed) braid group. All these structures are seen to coincide with the ones obtained from $(X,\Delta,q)$ up to terms containing $\hbar$. It is highly desirable to obtain more deformations, especially for the ternary case directly. Further computations are needed still. 

Finally, we conclude observing that the deformations of Example~\ref{ex:deformed2d} and Example~\ref{ex:Heisenberg} produce perturbations of the associated braid group representations. It is not currently clear whether these deformations perturb the link invariants as well. We proceed as follows. Let $R_\hbar = R + \hbar A$ denote the deformed Yang-Baxter operator. The matrix of $R_\hbar$ is invertible since $det(R_\hbar) = det(R) + \hbar p(\hbar)$ for some polynomial $p(\hbar)$ in $\hbar$, and this is a unit of $\mathbbm k [[\hbar]]/(\hbar^2)$. In fact, $det(R)$ is a unit of $\mathbbm k$, as a consequence of the invertibility of $R$, and $det(R)^{-1} - det(R)^{-2}\hbar p(\hbar)$ is the required inverse of $det(R_\hbar)$. In fact, one can construct a reverse to $q'$ as in Section~\ref{sec:LieTSD}, but this would require an unnecessary amount of further computations.

\subsection{A few concluding remarks}

We conclude this article by pointing out a few facts that might serve as possible directions of investigation in future work. 

First, we observe that the YB operators arising from $3$-Lie algebras, i.e. ternary operations that do not factor as composition of binary operations, have trivial twists. This is a direct consequence of the fact that the TSD operations used to construct the YB operators do not contain binary terms. These are, in fact, the terms that determine the non-triviality of twists in the case of TSD operations that are obtained by composing binary operations. If we interpret the composition of binary brackets as a ternary bracket, it is natural to ask whether the construction of the present article works when starting with a $3$-homotopy Lie algebra, i.e. a homotopy Lie algebra where brackets of dimension higher than $3$ vanish. We wonder whether the YB operators obtained following this proceedure would retain the characteristic of having non-trivial twists, as in the case of composed binary brackets, due to the appearance of binary operations that accompany the ternary bracket. 

As previously mentioned, the theory of deformations can be rephrased in terms of a cohomology theory for binary and ternary SD objects. See for instance \cite{EZ} for a general version of these in symmetric monoidal categories. Current work by the second author of this article shows that if $\Delta$ is not deformed, it is possible to construct a map from the second cohomology group of a Lie algebra with coefficients in the algebra itself, to the SD second cohomology group of the self-distributive object associated to the Lie algebra. It is not yet clear whether such a construction is possible when the comultiplication is deformed as well, as considered above. Very few direct computations are known to us, and having a principled way of constructing cocycles would be of great use, as they produce deformed representations of the (framed) braid group, and might consequently result as well in deformed invariants. Moreover, it is known that algebras can be deformed infinitely many times for all degrees of $\hbar$. It is of great interest to produce such deformations of higher order for the SD objects of this article, as they would provide invariants that are series in $\hbar$.

In addition, it is relevant to understand whether the undeformed invariant coincides with the trace of the transposition map in general, and whether the deformed invariants are related to previously known invariants. At present, we have very few computational indications of the former fact, and no understanding of the latter. We mention, in conclusion, that the invariants associated to deformed structures by means of SD cohomology constitute a Lie theoretic version of the well known cocycle invariants of \cite{CJKLS}.

\appendix 

\section{ Proofs of Lemmas~\ref{lem:deformedcoass} and \ref{lem:deformedSD}}

\begin{proof}[Proof of Lemma~\ref{lem:deformedcoass}]
	
	We need to verify that $\Delta'$ is coassociative. We have
	\begin{eqnarray*}
	(\mathbbm 1\otimes \Delta') \Delta' = (\mathbbm 1\otimes \Delta) \Delta + \hbar (\mathbbm 1\otimes \Delta)\psi + (\mathbbm 1\otimes \psi)\Delta + \hbar^2 (\mathbbm 1\otimes \psi)\psi \\
		(\Delta'\otimes \mathbbm 1) \Delta' = 	(\Delta\otimes \mathbbm 1)  \Delta + \hbar (\Delta \otimes \mathbbm 1)\psi + (\psi \otimes \mathbbm 1)\Delta + \hbar^2 (\psi \otimes \mathbbm 1)\psi.
	\end{eqnarray*}
	Equating the two results, using the fact that $\Delta$ is coassociative and considering that $\hbar^2 = 0$, we obtain the first equation of the statement of the lemma. The second set of equations comes similarly by imposing the counit axiom on $\Delta'$ and $\epsilon$. 
\end{proof}

\begin{proof}[Proof of Lemma~\ref{lem:deformedSD}]
	We need to impose the self-distributivity condition on $q' = q + \hbar \phi$. Omitting terms of quadratic or higher order for $\hbar$, we have for the LHS and RHS of self-distributivity
	\begin{eqnarray*}
	q'(q'\otimes \mathbbm 1) = q(q\otimes \mathbbm 1) + \hbar q(\phi\otimes \mathbbm 1) + \hbar \phi(q\otimes \mathbbm 1),
	\end{eqnarray*}
	and 
	\begin{eqnarray*}
	\lefteqn{q'(q'\otimes q')\shuffle (\mathbbm 1^{\otimes 2}\otimes \Delta')} \\
	&=&  q(q\otimes q)\shuffle (\mathbbm 1^{\otimes 2}\otimes \Delta) + \hbar q[\phi\otimes q+q\otimes \phi]\shuffle \Delta\\
	&& + \hbar \phi(q\otimes q)\shuffle \Delta + \hbar q(q\otimes q)\shuffle \psi.
	\end{eqnarray*}
	Equating them and using the fact that $q$ is self-distributive, we obtain the first equation of the lemma. The second equation is obtained following the same procedure, to impose the condition that $q'$ is a coalgebra morphism with respect to $\Delta'$.  
	\end{proof}

    \section{ Computations for Example~\ref{ex:deformed2d} }
    
    \begin{proof}[Computations for Example~\ref{ex:deformed2d}]
    	First, we need to show that $\Delta' = \Delta + \hbar \psi$ is coassociative. We do so by considering the first equation in Lemma~\ref{lem:deformedcoass} on basis vectors $(1,0)$. On $(1,0)$ and $(0,e_2)$ both sides of the equation vanish, and the equation holds. On $(0,e_1)$ we compute 
    	\begin{eqnarray*}
    	\lefteqn{(\psi\otimes \mathbbm 1)\Delta(0,e_1) + (\Delta\otimes \mathbbm 1)\psi(0,e_1)} \\
    	&=& \alpha[(0,e_1)\otimes (1,0) \otimes (0,e_2) + (1,0)\otimes (0,e_1) \otimes (0,e_2) \\
    	&& + (0,e_2)\otimes (1,0)\otimes (0,e_1) + (1,0)\otimes (0,e_2)\otimes (0,e_1)\\
    	&& + (0,e_1)\otimes (0,e_2)\otimes (1,0) + (0,e_2)\otimes (0,e_1)\otimes (1,0)],
    	\end{eqnarray*}
    	and 
    	\begin{eqnarray*}
    	\lefteqn{(\mathbbm 1\otimes \psi)\Delta (0,e_1) + (\mathbbm 1\otimes \Delta)\psi(0,e_1) }\\
    	&=& \alpha[(0,e_1)\otimes (0,e_2)\otimes (1,0) + (0,e_1)\otimes (1,0) \otimes (0,e_2) \\
    	&& + (0,e_2)\otimes (0,e_1)\otimes (1,0) + (0,e_2)\otimes (1,0)\otimes (0,e_1)\\
    	&& + (1,0)\otimes (0,e_1) \otimes (0,e_2) + (1,0)\otimes (0,e_2)\otimes (0,e_1)],
    	\end{eqnarray*}
    which are seen to coincide. The computation on $(0,e_2)$ is similar. It is also clear that $\Delta'$ is compatible with the counit, since $\psi$ maps into $\frak g_2\otimes \frak g_2$, over which $(\mathbbm 1\otimes \epsilon)$ and $(\epsilon \otimes \mathbbm 1)$ both vanish. Then, we have to verify that the conditions of Lemma~\ref{lem:deformedSD} are verified. The first equation reduces to $q(q\otimes q)\shuffle \psi$, since from the Jacobi identity (or more generally the $2$-cocycle condition) it follows that $q(\phi\otimes \mathbbm 1) + \phi(q\otimes \mathbbm 1) = q[\phi\otimes q+q\otimes \phi]\shuffle \Delta + \phi(q\otimes q)\shuffle \Delta$. But since $\frak g_2$ is solvable, we have that the iterated bracket corresponding to the composition $q(q\otimes q)\shuffle \psi = 0$  vanishes, and therefore $q'$ satisfies the self-distributive condition with respect to $\Delta'$. It is left to show that $q'$ and $\Delta'$ are compatible, i.e. that the second equality in Lemma~\ref{lem:deformedSD} holds. This is a direct (and tedious) computation on the basis vectors. Let us set $\Phi = \Delta \phi + \psi q$ and $\Lambda =   [\phi\otimes q + q\otimes \phi](\mathbbm 1\otimes \tau \otimes \mathbbm 1)(\Delta\otimes \Delta) + (q\otimes q)(\mathbbm 1\otimes \tau \otimes \mathbbm 1)(\psi\otimes \Delta + \Delta\otimes \psi)$ to shorten notation. We have 
    \begin{eqnarray*}
    	\Phi(1,0)\otimes (1,0) = \Lambda (1,0)\otimes (1,0) = 0,
    \end{eqnarray*} 
	 as well as 
	  \begin{eqnarray*}
	  	\Phi(1,0)\otimes (0,e_i) = \Lambda (1,0)\otimes (0,e_i) = 0.
	  \end{eqnarray*} 
  	Then, we have 
  	\begin{eqnarray*}
  	\Phi(0,e_i)\otimes (1,0) = \alpha_i[\psi(0,e_i) = (0,e_1)\otimes (0,e_2) - (0,e_2)\otimes (0,e_1)]
  	\end{eqnarray*}
  	which coincides with 
  	\begin{eqnarray*}
  	\lefteqn{\Lambda (0,e_i)\otimes (1,0)}\\
  	& = & [\phi\otimes q + q\otimes \phi][(0,e_i)\otimes (1,0)\otimes (1,0)\otimes (1,0)\\
  	&& + (1,0)\otimes (1,0)\otimes (0,e_i) \otimes (1,0)]\\
  	&& + (q\otimes q)[\alpha_1 (0,e_1)\otimes (1,0)\otimes (0,e_2)\otimes (1,0) \\
  	&& - \alpha_1 (0,e_2)\otimes (1,0)\otimes (0,e_1)\otimes (1,0) ].
  	\end{eqnarray*}
	On tensors $(0,e_i)\otimes (0,e_j)$ one computes 
	$$
	\Phi(0,e_i)\otimes (0,e_j) = \begin{cases}
	0 \ \ {\rm if} \ \ i = j\\
	(-1)^{i+1}\{\alpha_1(0,e_1)\otimes (0,e_2) + \alpha_1(0,e_2)\otimes (0,e_1)\\
	+  (0,e_1)\otimes (1,0) + (1,0)\otimes (0,e_1)\}\ \ {\rm otherwise }
	\end{cases}
	$$
	We also compute 
	\begin{eqnarray*}
	\lefteqn{\Lambda(0,e_i)\otimes (0,e_j)} \\
	&=& (q\otimes q)[\alpha_i (0,[e_i,e_j])\otimes (0,e_2) + \alpha_i(0,e_1)\otimes (0,[e_2,e_j])\\
	&& -\alpha_i (0,[e_2,e_j])\otimes (0,e_1) - \alpha_i (0,e_2)\otimes (0, [e_1,e_j]) \\
	&& + (0,[e_i,e_j]) \otimes (1,0) + (1,0)\otimes (0,[e_i,e_j]). 
	\end{eqnarray*}
	The latter is seen to coincide with the former for $i,j = 1,2$, which concludes the computation. To verify the claim regarding characteristic $2$, one proceeds in the same way. The term $\beta(0,e_2)\otimes (0,e_2)$ appears only in the last equation with a factor of $2$, and will therefore vanish. 
	\end{proof}

	\end{document}